\documentclass[reqno, oneside]{amsart}
\usepackage{etex}
\usepackage{hyperref}

\usepackage{chemarr}
\usepackage{amssymb}
\usepackage{comment}

\usepackage[a4paper]{geometry}

\theoremstyle{definition}
\newtheorem{nul}{}[section]
\newtheorem{dfn}[nul]{Definition}

\newtheorem{rmk}[nul]{Remark}

\newtheorem{cnstr}[nul]{Construction}
\newtheorem{cnv}[nul]{Convention}

\newtheorem{exm}[nul]{Example}

\newtheorem{qst}{Question}
\newtheorem*{dfn*}{Definition}
\newtheorem*{axm*}{Axiom}
\newtheorem*{ntn*}{Notation}
\newtheorem*{exm*}{Example}
\newtheorem*{exr*}{Exercise}
\newtheorem*{int*}{Intuition}
\newtheorem*{qst*}{Question}
\newtheorem*{rmk*}{Remark}

\theoremstyle{plain}

\newtheorem{claim}[nul]{Claim}
\newtheorem{thm}[nul]{Theorem}
\newtheorem{prop}[nul]{Proposition}
\newtheorem{lem}[nul]{Lemma}

\newtheorem{cor}{Corollary}[nul]
\newtheorem*{thm*}{Theorem}
\newtheorem*{prop*}{Proposition}
\newtheorem*{cor*}{Corollary}
\newtheorem*{lem*}{Lemma}
\newtheorem*{cnj*}{Conjecture}

\DeclareMathOperator{\Hom}{\text{Hom}}

\DeclareMathOperator{\smsh}{\wedge}

\DeclareMathOperator{\Z}{\mathbb{Z}}
\DeclareMathOperator{\E}{\mathbb{E}}

\DeclareMathOperator{\pic}{\mathrm{pic}}
\DeclareMathOperator{\Pic}{\mathrm{Pic}}

\usepackage{tikz}
\usetikzlibrary{matrix,arrows,decorations}
\usepackage{tikz-cd}

\usepackage{adjustbox}

\begin{document}

\title{Exotic Multiplications on Periodic Complex Bordism}
\author{Jeremy Hahn}
\address{Department of Mathematics, Massachusetts Institute of Technology, Cambridge, MA 02139}
\email{jhahn01@mit.edu}

\author{Allen Yuan}
\address{Department of Mathematics, Massachusetts Institute of Technology, Cambridge, MA 02139}
\email{alleny@mit.edu}

\begin{abstract}
Victor Snaith gave a construction of periodic complex bordism by inverting the Bott element in the suspension spectrum of $BU$.  This presents an $\mathbb{E}_\infty$ structure on periodic complex bordism by different means than the usual Thom spectrum definition of the $\mathbb{E}_\infty$-ring $MUP$.  Here, we prove that these two $\mathbb{E}_\infty$-rings are in fact different, though the underlying $\mathbb{E}_2$-rings are equivalent.  Nonetheless, we prove that both rings $\mathbb{E}_\infty$-orient $KU_2^{\wedge}$ and other forms of $K$-theory.
\end{abstract}


\setcounter{tocdepth}{1}
\maketitle

\tableofcontents

\vbadness 5000


\section{Introduction}

The complex bordism spectrum $MU$ occupies a fundamental place in modern homotopy theory.  
For instance, the field of chromatic homotopy theory is centered around an observation of Quillen which connects homotopy commutative ring maps out of $MU$ to formal group laws; accordingly, one is led to study \emph{complex oriented} cohomology theories, or ring spectra $E$ equipped with a map $MU \longrightarrow E$ of homotopy rings.  However, $MU$ is not just a homotopy commutative ring spectrum, but an $\E_{\infty}$-ring spectrum.  One can therefore ask:

\begin{qst}\label{qst:1}
Which $\mathbb{E}_\infty$-ring spectra $E$ receive $\mathbb{E}_\infty$-ring homomorphisms $MU \longrightarrow E$?
\end{qst}

This question has received significant attention in recent years.  The case of $K$-theory was studied in \cite{Walker, Moll,HopLaw}.  Ando \cite{Ando} shows that such a structured orientation on $E$ imposes a strong condition on the formal group law of $E$; work of Zhu \cite{Zhu} verifies this condition when $E$ is a Morava $E$-theory.  Hopkins and Lawson \cite{HopLaw} describe a series of obstructions, starting with Ando's condition, to producing $\E_{\infty}$-orientations.  On the other hand, Johnson and Noel show that such an orientation often implies that the formal group law on $\pi_*(E)$ is not in the image of the Quillen idempotent \cite{JohnsonNoel}.

In this paper, we study the variant of Question \ref{qst:1} for a variant of $MU$ known as the \emph{periodic} complex bordism spectrum $MUP$, which is the Thom spectrum of the tautological bundle over $BU\times \mathbb{Z}$:

\begin{qst}\label{qst:2}
Which $\mathbb{E}_\infty$-ring spectra $E$ receive $\mathbb{E}_\infty$-ring homomorphisms $MUP \longrightarrow E$?
\end{qst}

\begin{rmk}
Structured orientations by $MUP$ appear briefly in \cite{HopLaw}.  
Such orientations are also closely related to the theory of $H_{\infty}^2$ structures, as studied in \cite{BMMS} (see Remark \ref{rmk:hinfty2}).  
\end{rmk}

To make sense of Question \ref{qst:2}, one must be very clear about how the $\mathbb{E}_\infty$-ring structure on $MUP$ is defined.  As the authors understand it, the usual definition of this structure, due to \cite[Remark IX.7.2]{LMS} (see also \cite[Remark 2]{HopLaw}), runs as follows:

Consider the unstable $J$-homomorphism, which is the symmetric monoidal functor
$$J:\coprod_{n} BU(n) \simeq \{ \text{Complex Vector Spaces}\}^{\simeq} \longrightarrow \text{Spaces}$$
that sends a $d$-dimensional complex vector space $V$ to its one-point compactification $S^V$, a $2d$-dimensional sphere.  After further composing with the suspension spectrum functor $$\Sigma^{\infty}:\text{Spaces} \longrightarrow \text{Spectra},$$ the $J$-homomorphism takes values in the subgroupoid of invertible spectra and their automorphisms, known as the Picard category of Spectra.  This implies that $J$ factors through the group completion of the category of complex vector spaces, giving a stable $J$ functor 
$$J:BU \times \mathbb{Z} \simeq \{ \text{Virtual Complex Vector Spaces} \}^{\simeq} \longrightarrow \text{Pic}(\text{Spectra}) \subset \text{Spectra}.$$
Since this stable $J$ functor is symmetric monoidal, its homotopy colimit, $MUP,$ acquires an $\mathbb{E}_\infty$-ring structure \cite{LMS, OmarToby, ABGHRinfinity}.

\begin{cnv}
Throughout this paper, whenever we refer to $MUP$ as an $\mathbb{E}_\infty$-ring spectrum, we always give it the $\mathbb{E}_\infty$-ring structure constructed above.
\end{cnv}

In contrast, consider the following alternate construction of the periodic complex bordism spectrum by Snaith \cite{SnaithOriginal}:

Since $BU \simeq \Omega^{\infty} \Sigma^2 ku$ is an infinite loop space, $\Sigma^{\infty}_+ BU$ is an $\mathbb{E}_\infty$-ring spectrum.  Given an element in the homotopy of an $\mathbb{E}_\infty$-ring spectrum, we may invert that element to obtain another $\mathbb{E}_\infty$-ring.  In particular, we may invert the suspension of the Bott element $\beta:S^2 \longrightarrow BU$, considered as an element in stable homotopy $\beta \in \pi_2(\Sigma^{\infty}_+ BU)$.

\begin{thm}[Snaith \cite{SnaithOriginal}]
There is an equivalence of homotopy commutative ring spectra
$$\Sigma^{\infty}_+ BU[\beta^{-1}] \simeq MUP.$$
\end{thm}

Our aim, when beginning this project, was to upgrade Snaith's theorem to an equivalence of $\mathbb{E}_\infty$-ring spectra.  We were thus surprised to obtain the following result, which is the main theorem of this paper:

\begin{thm} \label{thm:main}
There is an equivalence of $\mathbb{E}_2$-ring spectra
$$\Sigma^{\infty}_+ BU[\beta^{-1}] \simeq MUP,$$
but there is \textbf{NOT} an equivalence of $\mathbb{E}_5$-ring spectra.  In particular, there is no equivalence of $\mathbb{E}_\infty$-rings.
\end{thm}

\begin{rmk}
In order to prove Theorem \ref{thm:main}, we will construct a certain $\mathbb{E}_\infty$-ring structure on $H\mathbb{Z}P$, by which we denote periodic integral homology.  By construction, there will be an $\mathbb{E}_\infty$-ring homomorphism $MUP \longrightarrow H\mathbb{Z}P$, but we will show that there can be no $\mathbb{E}_5$-ring homomorphism $\Sigma^{\infty}_+ BU[\beta^{-1}] \longrightarrow H\mathbb{Z}P$.  As we will explain, one interpretation of our result is that `no choice of total Chern class can be a five-fold loop map.'  This is related to another result of Snaith \cite{SnaithNotMultiplicative}, which states roughly that `the usual choice of total Chern class is not an infinite loop map.'
\end{rmk}

Despite the plethora of work that has gone into understanding the standard multiplication on $MUP$ \cite{Ando,HopLaw}, there are some reasons to prefer the `exotic' multiplication that is $\Sigma^{\infty}_+ BU[\beta^{-1}]$.  Perhaps the main reason is the latter's clear connection to topological complex $K$-theory, as mediated through another theorem of Snaith:

\begin{thm}[Snaith \cite{SnaithOriginal}]
There is an equivalence of homotopy commutative ring spectra 
$$\Sigma^{\infty}_+ \mathbb{CP}^{\infty}[\beta^{-1}] \simeq KU.$$
\end{thm}

\begin{rmk}
Unlike with periodic complex bordism, the above equivalence can be lifted to one of $\mathbb{E}_\infty$-ring spectra.  This follows immediately from \cite[Theorem 6.2]{BakerRichter}, and one can also deduce it from the existence of a symmetric monoidal inclusion functor
$$\mathbb{CP}^{\infty} \simeq \{\text{Complex Lines},\otimes \}^{\simeq} \longrightarrow \{\text{Virtual Complex Vector Spaces},\otimes \}^{\simeq} \simeq BU \times \mathbb{Z}.$$
\end{rmk}

Since the determinant map $$BU \longrightarrow BU(1)\simeq \mathbb{CP}^{\infty}$$ is an infinite loop map, one has the following theorem:

\begin{thm}[Snaith]
There is a map of $\mathbb{E}_\infty$-ring spectra
$$\Sigma^{\infty}_+ BU[\beta^{-1}] \longrightarrow \Sigma^{\infty}_+ \mathbb{CP}^{\infty}[\beta^{-1}] \simeq KU.$$
\end{thm}

This construction of an $\mathbb{E}_\infty$-orientation of $KU$ is so canonical that it can be ported into other settings, such as motivic homotopy theory \cite{GepnerSnaith}.  In contrast, while it seems to be folklore that there is an $\mathbb{E}_\infty$-ring homomorphism $$MUP \longrightarrow KU,$$ the authors could find no proof of this in the literature.  We provide in the final section of this document a computational proof of the slightly weaker result that there is an $\mathbb{E}_\infty$-ring map from $MUP$ into the $2$-completion of $KU$.  We prove this as part of a more general story.

\begin{dfn}
A \textit{form of $K$-theory} (at the prime $2$) is a Morava $E$-theory corresponding to a height $1$ formal group over $\mathbb{F}_2$. \cite{MorKThy,Charmaine}
\end{dfn}

\begin{rmk}
Forms of $K$-theory are $\mathbb{E}_\infty$-ring spectra.  They are $2$-periodic, with each odd homotopy group isomorphic to $0$ and each even homotopy group isomorphic to the $2$-adic integers $\mathbb{Z}_2$.  Choosing the multiplicative formal group over $\mathbb{F}_2$ yields a form of $K$-theory that is the $2$-completion $KU^{\wedge}_2$ of topological complex $K$-theory.
\end{rmk}

\begin{thm} \label{Thm:FormOrientation}
Let $E$ denote a form of $K$-theory.  Then there is both an $\mathbb{E}_\infty$-ring homomorphism
$$MUP \longrightarrow E$$
and an $\mathbb{E}_\infty$-ring homomorphism
$$\Sigma^{\infty}_+ BU[\beta^{-1}] \longrightarrow E.$$
\end{thm}

\begin{rmk} We believe that slight variants of our arguments work just as well at odd primes $p>2$.  We work at the prime $2$ for simplicity.  Our arguments do not provide orientations of Morava $E$-theories $E_{(k,\mathbb{G})}$ associated to larger characteristic $2$ fields $k$, unless the formal group $\mathbb{G}$ over $k$ is isomorphic to one pushed forward from a height $1$ formal group over $\mathbb{F}_2$.
\end{rmk}

\begin{rmk} 
Highly structured $MU$ (as opposed to $MUP$) orientations of forms of $K$-theory have been constructed in several locations \cite{Moll,HopLaw}, with the case of $p$-adic $K$-theory receiving particular attention in \cite{Walker}.
\end{rmk}

A number of open questions are raised by our work here.

\begin{qst} Is there an $\mathbb{E}_4$-equivalence $\Sigma^{\infty}_+ BU[\beta^{-1}] \simeq MUP$?  The authors find such an equivalence very believable, but find difficulties proving it that are related to the difficulties Chadwick and Mandell \cite{ChadwickMandell} encounter when trying to check whether the Quillen idempotent on $MU$ is $\mathbb{E}_4$.
\end{qst}

\begin{qst} Is there an $\mathbb{E}_\infty$-ring homomorphism $MU \longrightarrow \Sigma^{\infty}_+ BU[\beta^{-1}]$?
\end{qst}

\begin{rmk} Tyler Lawson has emphasized to us that there are yet other $\mathbb{E}_\infty$-ring spectra that might naturally be called periodic complex bordism, such as the Tate spectrum $MU^{tS^1}$.  It would be enlightening to have a catalogue of various $\mathbb{E}_\infty$ `forms' of periodic complex bordism, as well as forms of periodic integral homology such as $H\mathbb{Z}^{tS^1}$.
\end{rmk}

\begin{qst}
In the sense of the previous remark, which forms of periodic complex bordism orient which forms of periodic integral homology?  Is there a form of periodic integral homology that is $\mathbb{E}_\infty$-oriented by $\Sigma^{\infty}_+ BU[\beta^{-1}]$?
\end{qst}

\begin{qst}
In addition to the famous open question of which Morava $E$-theories are $\mathbb{E}_\infty$-oriented by $MU$ \cite{HopLaw}, which Morava $E$-theories are $\mathbb{E}_\infty$-oriented by which forms of periodic complex bordism?
\end{qst}

\subsection{Acknowledgements}
The authors are grateful to David Gepner, Tyler Lawson, Haynes Miller, Eric Peterson, and Andrew Senger for numerous helpful conversations.  They would like to extend special thanks to Mike Hopkins and Jacob Lurie for  their numerous insights and constant encouragement.  The authors were supported by the NSF under Grants DMS-1803273 and DGE-1122374.

\section{Preliminaries} \label{sec:prelim}

 \subsection{Periodic Thom Spectra}

\begin{dfn}\label{dfn:mup}
Let $MUP$ denote the $\E_{\infty}$ Thom spectrum \cite{LMS} associated to the complex $J$-homomorphism $$J: ku \to \pic \mathbb{S}.$$  There is a canonical element $u\in \pi_2(MUP)$ such that $\pi_*(MUP) = \pi_*(MU)[u^{\pm 1}]$.  
\end{dfn}

\begin{rmk}\label{rmk:mupe2split}
Applying $\Omega^2$ to the inclusion $\mathbb{CP}^{\infty} \to BU$ yields a double loop map $\mathbb{Z} \to BU \times \mathbb{Z}$. The Thom spectrum of this map $\Z \to BU \times \Z$ is an $\E_2$-ring spectrum $$P \simeq \bigvee_{i\in \mathbb{Z}} S^{2i}.$$  This yields an identification of $\E_2$-rings $MUP \simeq MU \smsh P.$  However, the spectrum $P$ cannot be given the structure of an $\mathbb{E}_{\infty}$-ring \cite[Proposition VII.6.1]{Hinfty}, and so there is no similar expression for $MUP$ as a smash product of $\mathbb{E}_\infty$-ring spectra.  
\end{rmk}

Note that $MUP$ admits an $\E_{\infty}$-ring map from $MU$ induced by the map of $\E_{\infty}$ spaces $BU \to BU \times \Z$.  Using this, Definition \ref{dfn:mup} allows one to define a periodic version of any $\E_{\infty}$ $MU$-algebra:

\begin{cnstr}
The natural truncation map $MU \to H\Z$ exhibits $H\Z$ as an $\E_{\infty}$-$MU$-algebra. 
Let $H\Z P$ denote the relative smash product $MUP \smsh_{MU} H\Z$.  By construction, $H\Z P$ is an $\E_{\infty}$-ring spectrum equipped with an $\E_{\infty}$-ring map $MUP \to H\Z P$.  
\end{cnstr}

\subsection{Complex orientations and total Chern classes}

Recall that the integral cohomology of $BU$ is given by $$H^*(BU;\Z) \cong \Z[[c_1,c_2,\cdots]].$$ 
This isomorphism allows one to define the Chern classes of any complex vector bundle $V$ over a space $X$.  One may assemble these Chern classes into a \emph{total Chern class} $c(V)$, which is a power series in the formal variable $u$: $$c(V) = 1+c_1(V) u+ c_2(V) u^2 + \cdots \in H^*(X;\Z)[[u]].$$  The total Chern class construction $V \mapsto c(V)$ takes addition of vector bundles to multiplication of power series.
 
\begin{dfn}
Let $E$ be a homotopy commutative ring spectrum. A \emph{complex orientation} of $E$ is a map of homotopy commutative ring spectra $MU \to E$. A \emph{periodic complex orientation} of $E$ is a map of homotopy commutative ring spectra $MUP \to E$. 
\end{dfn}


\begin{rmk}  By Remark \ref{rmk:mupe2split}, we have an identification of homotopy commutative rings $MUP \simeq MU \smsh P$.  Since the smash product is the coproduct in the category of homotopy commutative rings, a periodic complex orientation of $E$ is the separate data of a complex orientation of $E$ and the additional choice of a unit in $u\in \pi_2 E$.
\end{rmk}


\begin{exm}\label{exm:buhzp}
By construction, there is a canonical periodic complex orientation $MUP \to H\Z P.$  
\end{exm}

Snaith \cite{SnaithOriginal} constructed an equivalence of homotopy commutative rings $\Sigma^{\infty}_+ BU[\beta^{-1}] \to MUP.$  Composing with the periodic complex orientation $MUP \to H\Z P$, we obtain a ring map $$\Sigma^{\infty}_+BU \to H \Z P$$ which is the same data as a map of $H$-spaces $$BU \to GL_1 H\Z P =  \Z/2 \times \prod_{i\geq 1} K(\Z , 2i).$$  Snaith's equivalence is constructed such that the components of this map correspond to the elements $c_i \in H^{2i}(BU; \Z ).$  Together, they assemble to the usual total Chern class $$c = 1+c_1 u + c_2 u^2 + \cdots \in H\Z P^0(BU).$$  The fact that the map $BU\to GL_1 H\Z P$ is a map of $H$-spaces encodes the fact that the total Chern class takes sums of vector bundles to multiplication of power series.  

Different choices of periodic complex orientation $MUP\to H\Z P$ beget different notions of total Chern class.  We can describe these other total Chern classes in terms of the above one as follows.  

By the splitting principle, a total Chern class is determined by its value on line bundles, which is the composite $$BU(1) \to BU \to GL_1H\Z P.$$  For the total Chern class constructed above, this is given by $1+ut$, where $t\in H^2(\mathbb{C}P^{\infty})$ is the generator.  Any other periodic complex orientation of $H\Z P$ arises from this one by changing the coordinate and/or choice of unit.  The unit is determined up to sign, and the coordinate can be changed by any series of the form $t \mapsto a_1 t + a_2 ut^2 + a_3 u^2t^3 + \cdots$ where $a_i \in \Z$ and $a_1 = \pm 1$ is a unit.  We can therefore associate to any ring map $MUP \to H\Z P$ its corresponding total Chern class for line bundles, which is a series $$r(t) = 1 + a_1 u t + a_2 u^2 t^2 + a_3 u^3 t^3 + \cdots \in H\Z P^0(\mathbb{C}P^{\infty}),$$ where $a_i\in \Z$ are integers and $a_1= \pm 1$ is a unit.

\subsection{Power operations on $MUP$}

Let $X$ be a space and let $E$ be an $\E_{\infty}$-ring spectrum.  Suppose that we are given a class $\gamma \in E^0(X)$, represented by a map $\gamma: \Sigma^{\infty}_+ X \to E$.  The composite
$$(\Sigma^{\infty}_+ X)_{hC_2} \to (\Sigma^{\infty}_+ (X\times X))_{hC_2} \to (E\smsh E)_{hC_2} \to E$$
determines a new cohomology class which we denote $P^2_E(\gamma) \in E^0(X\times \mathbb{R}P^{\infty})$ and refer to as the \emph{total power operation} on $\gamma$.   This construction has the following elementary properties:

\begin{lem} [\cite{BMMS}, Corollary II.1.6] \label{lem:powopbasic}
Let $X, E$ as above and $a,b\in E^0(X)$.  Then the following formulas hold:
\begin{enumerate}
\item $P^2_{E}(ab) = P^2_E(a) P^2_E(b).$
\item $P^2_{E}(a+b) = P^2_E(a) + P^2_E(b) + \mathrm{tr}(ab)$ where $\mathrm{tr}:E^0(X) \to E^0(X\times \mathbb{R}P^{\infty})$ is induced by the stable transfer $\Sigma^{\infty}_+ BC_2 \to \mathbb{S}.$  
\end{enumerate}
\end{lem}

We will be concerned with the special case $X=\mathbb{C}P^{\infty}$, $E=MUP$, and where $\gamma = ut \in MUP^0(\mathbb{C}P^{\infty})$ is the composite $$\mathbb{C}P^{\infty} \xrightarrow{t} \Sigma^2MU \xrightarrow{u} MUP$$ of the canonical class $t\in MU^2(\mathbb{C}P^{\infty})$ with the periodicity element.  We have that $$MU^*(\mathbb{C}P^{\infty} \times \mathbb{R}P^{\infty}) = MU_*[[t,z]]/([2]_{MU}(z)),$$ where $z\in MU^2(\mathbb{R}P^{\infty})$ is the pullback of $t$ under the natural map $\mathbb{R}P^{\infty} \to \mathbb{C}P^{\infty}$.  Analogously,  $$MUP^*(\mathbb{C}P^{\infty}\times \mathbb{R}P^{\infty}) \simeq MU_*[[t, z]][u^{\pm 1}]/([2]_{MU}(z)).$$  Quillen essentially computed the total power operation in $MUP$ cohomology:

\begin{prop}[\cite{Quillen}]\label{prop:quillencomp}
The total power operation on $ut\in MUP^0(\mathbb{C}P^{\infty})$ is given by 
$$P^2(ut) = u^2t(t+_{MU} z) \in MUP^0(\mathbb{C}P^{\infty} \times \mathbb{R}P^{\infty}).$$
\end{prop}

\begin{proof}
Consider the diagram 

\begin{equation*}
\begin{tikzcd}
\Sigma^{\infty}_+(\mathbb{C}P^{\infty}\times \mathbb{R}P^{\infty}) \arrow[rdd, "\tilde{P}^2_{MU}(t)", swap, dashed ]\arrow[r]&\Sigma^{\infty}_+(\mathbb{C}P^{\infty}\times \mathbb{C}P^{\infty})_{hC_2} \arrow[d,"D_2(t)"]  &\\
 &(\Sigma^2 MU \smsh \Sigma^2 MU)_{hC_2}\arrow[d,"\alpha"]\arrow[r] &(MUP \smsh MUP )_{hC_2} \arrow[d]\\
 &\Sigma^4 MU\arrow[r,"u^2"] & MUP.\\
\end{tikzcd}
\end{equation*}

Here, the bottom right square is defined to be the Thomification of the commutative square
\begin{equation*}
\begin{tikzcd}
(BU\times BU)_{hC_2} \arrow[r, "i_1\times i_1"] \arrow[d] & ((BU\times \Z)\times (BU\times \Z))_{hC_2} \arrow[d] &\\
BU \arrow[r,"i_2"] & BU\times \Z \arrow[r,"J"]& \mathrm{Pic}(\mathbb{S})\\
\end{tikzcd}
\end{equation*}
where the vertical maps are induced by direct sum of vector spaces and $i_n:BU \to BU\times \Z$ denotes the inclusion of the component corresponding to complex vector spaces of virtual dimension $n$.   The dashed composite, which we have named $\tilde{P}^2_{MU}(t)$, was computed by Quillen \cite{Quillen} (see also \cite[Theorem 7.12]{Rudyak} and \cite[Lemma 2.5.6]{Eric}) to be $t(t+_{MU}z)$.  Since the upper long composite is the definition of the total power operation on $ut \in MUP^0(\mathbb{C}P^{\infty})$, the result follows.

\end{proof} 

\begin{rmk}\label{rmk:hinfty2}
The proof essentially says that the $\E_{\infty}$ structure on $MUP$ extends the classical $H_{\infty}^2$ structure on $MU$ \cite[Corollary VIII.5.3]{BMMS}.  In the language of \cite[Problem 1.3.7]{Tyler}, the map $\Z \to \Pic(MU)$ defining $MUP$ refines the $H_{\infty}^2$ structure on $MU$ to an $\E_{\infty}^2$ structure.  
\end{rmk}

Proposition \ref{prop:quillencomp} also allows us to compute power operations in $\E_{\infty}$-ring spectra which receive an $\E_{\infty}$-ring map from $MUP$.  In particular, using the fact that $H\Z$ carries the additive formal group law, we deduce the following formula for power operations in $H\Z P$.  

\begin{cor}
Let $ut\in H\Z P^0(\mathbb{C}P^{\infty})$ be the class coming from the complex orientation.  Then $$P^2_{H\Z P}(ut) = u^2t(t+z).$$
\end{cor}

In addition, we will record the following basic lemma about power operations in $H\Z P$:
\begin{lem}\label{lem:hzpbasic}
The following formulas hold:
\begin{enumerate}
\item Let $X$ be a space and $a,b\in H\Z P^0(X)$.  Then $$P^2_{H\Z P}(a+b) = P^2_{H\Z P}(a) + P^2_{H\Z P}(b) + 2ab,$$ where $2ab$ refers to its image under the natural map $H\Z P^0(X) \to H\Z P^0(X \times \mathbb{R}P^{\infty})$ induced by projection.  
\item The power operation on $n\in \pi_0(H\Z P) \simeq \Z$ is given by $P^2_{H\Z P}(n) = n^2$.  
\end{enumerate}
\end{lem}
\begin{proof}
The first statement follows from the first statement of Lemma \ref{lem:powopbasic} by observing that the transfer in $H\Z P$ is just multiplication by $2$.  The second statement follows by using $P^2_{H\Z P}(1) =1$ and inducting using the first statement.  
\end{proof}

\section{Obstructions} \label{sec:PowOp}

We now turn to the proof of Theorem \ref{thm:main}.  The first half follows from:

\begin{prop}
Any map of homotopy commutative rings $$\Sigma^{\infty}_+ BU[\beta^{-1}] \to MUP$$ can be promoted to an $\E_2$-map.  
\end{prop}
\begin{proof}
Let $\Sigma^{\infty}_+ BU[\beta^{-1}] \to MUP$ be a map of homotopy commutative rings.  The resulting map $$\Sigma^{\infty}_+ BU \to \Sigma^{\infty}_+ BU[\beta^{-1}] \to MUP$$ lifts to an $\E_2$-ring map by \cite[Theorem 6.1]{HY} and has the property that the Bott element is inverted.  It therefore extends to an $\E_2$-ring map $\Sigma^{\infty}_+BU[\beta^{-1}] \to MUP$ agreeing with the original map by \cite[Theorem A.1]{Akhil}.
\end{proof}

For the second half, we note that if there was indeed an equivalence of $\mathbb{E}_{5}$-ring spectra $\Sigma^{\infty}_+BU[\beta^{-1}] \simeq MUP$, then one would obtain an $\mathbb{E}_{5}$-ring map $\Sigma^{\infty}_+BU[\beta^{-1}] \to H\Z P$.  Thus, Theorem \ref{thm:main} follows from the following stronger claim:

\begin{thm}
There does not exist a map of $\mathbb{E}_{5}$-rings $\Sigma^{\infty}_+BU[\beta^{-1}] \to H\Z P$.  
\end{thm}

\begin{proof}
Let us suppose that such a map existed, arising from a map $r:\Sigma^{\infty}_+ BU \to H\Z P$.  By the discussion following Example \ref{exm:buhzp}, we obtain a total Chern class for line bundles $$r(t) = 1 + a_1 ut + a_2 u^2t^2 + a_3 u^3 t^3 + \cdots \in H\Z P^0 (\mathbb{C}P^{\infty})$$ where $a_i\in \Z$ and $a_1 = \pm 1$.  

\begin{ntn*}
For $n \in \mathbb{N}\cup \{ \infty\}$, let $\text{Config}_2(\mathbb{R}^{n})$ denote the ordered configuration space of two points in $\mathbb{R}^n$.  This space admits a natural action of $C_2$ with homotopy quotient $\mathbb{R}P^{n-1}$.  Recall that an $\E_n$-ring spectrum $R$ has multiplication parametrized by configurations of points in $\mathbb{R}^n$, and in particular, comes with the structure of a map $(R \smsh R \smsh \Sigma^{\infty}_+ \text{Config}_2(\mathbb{R}^{n}))_{hC_2} \to R$.  
\end{ntn*}

Note that the composite $\mathbb{C}P^{\infty} = BU(1) \to BU \to BU\times BU$ of the natural inclusion followed by the diagonal induces a commutative square
\begin{equation}\label{dia:config}
\begin{tikzcd}
\mathbb{C}P^{\infty} \times \text{Config}_2(\mathbb{R}^{5}) \arrow[r]\arrow[d] &  BU\times BU \times \text{Config}_2(\mathbb{R}^{5}) \arrow[d]\\
\mathbb{C}P^{\infty} \times \text{Config}_2(\mathbb{R}^{\infty}) \arrow[r]&  BU\times BU \times \text{Config}_2(\mathbb{R}^{\infty}),
\end{tikzcd}
\end{equation}
of $C_2$-equivariant spaces.  For the next step, it is important to note that the homotopy orbit space $(\mathbb{C}P^{\infty} \times \text{Config}_2(\mathbb{R}^{5}))_{hC_2} \simeq \mathbb{C}P^{\infty} \times (\text{Config}_2(\mathbb{R}^5)_{hC_2})$ is homotopy equivalent to $\mathbb{C}P^{\infty} \times \mathbb{R}P^4$.  
We will prove the theorem by examining the following diagram, where the top left square is obtained by applying $(-)_{hC_2}$ to (\ref{dia:config}) and the right half of the diagram is induced by $r$:

\begin{equation*}
\begin{tikzcd}
\Sigma^{\infty}_+ (\mathbb{C}P^{\infty} \times \mathbb{R}P^{4}) \arrow[r] \arrow[d] & (\Sigma^{\infty}_+ BU \smsh \Sigma^{\infty}_+ BU\smsh \Sigma^{\infty}_+\text{Config}_{2}(\mathbb{R}^5))_{hC_2} \arrow[r]\arrow[d] & (H\Z P \smsh H\Z P\smsh \Sigma^{\infty}_+\text{Config}_2(\mathbb{R}^5))_{hC_2} \arrow[d] \\
\Sigma^{\infty}_+  (\mathbb{C}P^{\infty}\times \mathbb{R}P^{\infty}) \arrow[r,"\delta"] & (\Sigma^{\infty}_+ BU \smsh \Sigma^{\infty}_+ BU)_{hC_2} \arrow[r,"D_2(r)"]\arrow[d,"\mu_{BU}"] & (H\Z P \smsh H\Z P)_{hC_2} \arrow[d,"\mu_{H\Z P}"] \\
 & \Sigma^{\infty}_+ BU \arrow[r,"r"] & H\Z P.\\  
\end{tikzcd}
\end{equation*}

The top left and top right squares commute by construction.  If the map $r: \Sigma^{\infty}_+ BU \to H\Z P$ were $\E_{\infty}$, then the bottom right square would also commute.  If that map were only $\E_{5}$, then the bottom right square need not commute, but we would still see that the long composites $\mathbb{C}P^{\infty} \times \mathbb{R}P^{4} \to H\Z P$ would agree.  We show that this is not the case.

Namely, we compute the maps $r\circ \mu_{BU}\circ \delta$ and $\mu_{H\Z P}\circ D_2(r) \circ \delta$ as elements in the cohomology group $$H\Z P^0 (\mathbb{C}P^{\infty}\times \mathbb{R}P^{\infty})\simeq \Z [[ut,uz]]/(2uz),$$ and then show that they do not agree upon restriction to $$H\Z P^0(\mathbb{C}P^{\infty}\times \mathbb{R}P^4) \simeq \Z[[ut,uz]]/(2uz, (uz)^3).$$  

\begin{ntn*}
For ease of notation, we set $\tilde{t} := ut$ and $\tilde{z} := uz$ for the remainder of the proof.  
\end{ntn*}

\textbf{Computation 1:} $r\circ \mu_{BU}\circ \delta \in  H\Z P^0 (\mathbb{C}P^{\infty}\times \mathbb{R}P^{\infty})$.  

This amounts to the computation of the total Chern class (as defined by $r$) of a certain vector bundle on $\mathbb{C}P^{\infty}\times \mathbb{R}P^{\infty}$.  To describe this vector bundle, note that there is a 
commutative diagram of spaces

\begin{equation*}
\begin{tikzcd}
\mathbb{C}P^{\infty}\arrow[d]\arrow[r]& \mathbb{C}P^{\infty} \times \mathbb{C}P^{\infty} \arrow[r]\arrow[d] & BU \times BU\arrow[d] \\
\mathbb{C}P^{\infty} \times \mathbb{R}P^{\infty}\arrow[r] &(\mathbb{C}P^{\infty} \times \mathbb{C}P^{\infty})_{hC_2} \arrow[r] & BU. 
\end{tikzcd}
\end{equation*}

The diagram classifies the bundle $(\mathcal{L}-1) \oplus (\mathcal{L}-1)$ over $\mathbb{C}P^{\infty}$, and the bottom row arises from observing that this bundle is equivariant with respect to the swap action.  Here, the left two vertical maps are given by projection to the homotopy orbits for $C_2$.  It follows that the bundle over $\mathbb{C}P^{\infty}\times \mathbb{R}P^{\infty}$ classified by $\mu_{BU}\circ \delta$ is $(\mathcal{L}-1)\otimes_{\mathbb{C}} \rho$, where $\rho$ denotes the complex regular representation of $C_2$ over $\mathbb{R}P^{\infty} = BC_2$ (for instance, see \cite[Lemma 7.10]{Rudyak}).  

Let $\sigma$ denote the complex sign representation, so that $\rho = \sigma \oplus 1$.   We then have a decomposition $$(\mathcal{L}-1)\otimes \rho = \mathcal{L}\otimes \sigma + \mathcal{L} - \sigma - 1.$$ 

Then the class of $r\circ \mu_{BU}\circ \delta$ is simply the total Chern class of this bundle, which we can compute as: 
\begin{align*}
r(\mathcal{L}\otimes \sigma + \mathcal{L} - \sigma - 1) &= \frac{r(\mathcal{L}\otimes \sigma)r(\mathcal{L})}{r(\sigma)} \\
&= \frac{r(t+z)r(t)}{r(z)}.\end{align*}

\textbf{Computation 2:} $\mu_{H\Z P}\circ D_2(r) \circ \delta \in  H\Z P^0 (\mathbb{C}P^{\infty}\times \mathbb{R}P^{\infty})$.  

This is the total power operation on the class in $H\Z P^0(\mathbb{C}P^{\infty})$ represented by the composite $$\Sigma^{\infty}_+ \mathbb{C}P^{\infty} = \Sigma^{\infty}_+ BU(1) \to \Sigma^{\infty}_+  BU  \xrightarrow{r} H\mathbb{Z} P.$$  In the notation of Section \ref{sec:prelim}, this is $P^2_{H\Z P}(r(t)) = P^2_{H\Z P}(1+a_1 \tilde{t} + a_2 \tilde{t}^2 + \cdots).$  By part (a) of Lemma \ref{lem:hzpbasic}, 
\begin{align*}P^2_{H\Z P}(1+a_1\tilde{t} +a_2\tilde{t}^2 + \cdots) &= P^2(1) + P^2(a_1\tilde{t}) + P^2(a_2\tilde{t}^2) + \cdots + (\text{cross terms}) \\
&= 1 + a_1^2 \tilde{t}(\tilde{t}+\tilde{z}) + a_2^2 \tilde{t}^2(\tilde{t}+\tilde{z})^2 + \cdots + (\text{cross terms}),\end{align*}  where we have used the fact that $P^2$ is multiplicative and $P^2_{H\Z P}(a)=a^2$ for $a \in \pi_0 H\Z P.$  

To see that the $\mathbb{E}_5$-ring map in the theorem cannot exist, it suffices to show the following:
\begin{claim}
The elements $\frac{r(t+z)r(t)}{r(z)}$ and $P^2_{H\Z P}(r(t))$ are not equal in $H\Z P^0(\mathbb{C}P^{\infty} \times \mathbb{R}P^{\infty})$.  In fact, they remain unequal under the map $$H\Z P^0(\mathbb{C}P^{\infty} \times \mathbb{R}P^{\infty}) \simeq \Z[[\tilde{t},\tilde{z}]]/(2\tilde{z}) \twoheadrightarrow \Z[[\tilde{t},\tilde{z}]]/(2\tilde{z}, \tilde{z}^3) \simeq H\Z P^0(\mathbb{C}P^{\infty} \times \mathbb{R}P^{4}).  $$
\end{claim}

Suppose the equation $$r(t+z)r(t) = P^2_{H\Z P}(r(t))r(z)$$ holds modulo $\tilde{z}^3$ and $2\tilde{z}$.  We explicitly compute part of the $\tilde{z}$ and $\tilde{z}^2$ coefficients.  First, since $\tilde{z}$ and $\tilde{z}^2$ are $2$-torsion in the ring and the cross terms in our computation for $P^2_{H\Z P}(r(t))$ were all divisible by $2$, we may safely ignore them in what follows.  We are therefore looking to compare the $\tilde{z}$ and $\tilde{z}^2$ coefficients, mod $2$, of the two expressions, $$r(t+z)r(t) = (1+ a_1 (\tilde{t}+\tilde{z}) + a_2(\tilde{t}+\tilde{z})^2 + \cdots)(1+a_1\tilde{t} +a_2 \tilde{t}^2 + \cdots)$$ and $$(1 + a_1^2 \tilde{t}(\tilde{t}+\tilde{z}) + a_2^2 \tilde{t}^2(\tilde{t}+\tilde{z})^2 + \cdots)(1+a_1\tilde{z}+a_2\tilde{z}^2 + \cdots).$$  

By comparing $\tilde{z}^2\tilde{t}$ coefficients mod $2$ (and using that $a_i^2 \equiv a_i \pmod 2$), we find the relation $$ a_1a_2 + a_3 = a_1.$$ Recall that $a_1= \pm 1$, and so $a_1 \equiv 1\pmod 2.$  This implies that $a_2 + a_3 \equiv 1\pmod 2.$

On the other hand, by comparing $\tilde{z}^2\tilde{t}^2$ coefficients, we find that $$a_2 + a_1a_3 = a_2 + a_1a_2.$$  This means that $a_2+ a_3 \equiv 0 \pmod 2$, which contradicts the previous condition.  This proves the claim and the theorem follows.   
\end{proof}

\section{Orientations}\label{sec:MUPKU}

The determinant homomorphism $BU \rightarrow \mathbb{CP}^{\infty}$ induces an $\mathbb{E}_\infty$-ring map
$$\Sigma^{\infty}_+ BU[\beta^{-1}] \longrightarrow \Sigma^{\infty}_+ \mathbb{CP}^{\infty}[\beta^{-1}] \simeq KU.$$
It seems more delicate to construct an $\mathbb{E}_\infty$-ring map
$$MUP \longrightarrow KU.$$
In fact, while the authors feel that such an orientation is `well-known,' we could not locate a reference in the literature.  We record in this section a simple homotopy-theoretic proof that there is an orientation of at least the $2$-completion of $KU$.  In fact, we will prove the stronger result that any form of $K$-theory in the sense of Morava \cite{MorKThy,Charmaine} receives an $\mathbb{E}_\infty$-$MUP$-orientation.

\begin{dfn}
A \textit{form} of $K$-theory (at the prime $2$) is a Morava $E$-theory corresponding to a height $1$ formal group law over $\mathbb{F}_2$.
\end{dfn}

\begin{rmk}
Forms of $K$-theory are $\mathbb{E}_\infty$-ring spectra.  They are $2$-periodic, with each odd homotopy group isomorphic to $0$ and each even homotopy group isomorphic to the $2$-adic integers $\mathbb{Z}_2$.  Choosing the multiplicative formal group law over $\mathbb{F}_2$ yields a form of $K$-theory that is the $2$-completion $KU^{\wedge}_2$ of topological complex $K$-theory.
\end{rmk}

The remainder of this section will consist of the proof of the following:

\begin{thm} \label{Thm:FormOrientation}
Let $E$ denote a form of $K$-theory.  Then there is both an $\mathbb{E}_\infty$-ring homomorphism
$$MUP \longrightarrow E$$
and an $\mathbb{E}_\infty$-ring homomorphism
$$\Sigma^{\infty}_+ BU[\beta^{-1}] \longrightarrow E.$$
\end{thm}

To produce either of the $\mathbb{E}_\infty$-ring maps promised by Theorem \ref{Thm:FormOrientation} it is helpful to record the following lemmas:
\begin{lem} \label{Lemma:AdamsPriddy}
Let $ku^{\wedge}_{2}$ denote the connective cover of $KU_2^{\wedge}$.  Then there are no non-trivial maps of spectra
$$ku \longrightarrow \Sigma ku^{\wedge}_{2},$$
$$ku \longrightarrow \Sigma^3 ku^{\wedge}_{2}, \text{ or}$$
$$ku \longrightarrow \Sigma^5 ku^{\wedge}_{2}.$$
\end{lem}

\begin{proof}
In \cite{AngLind} it is remarked that the Adams spectral sequence
$$Ext^{s,t}_{\mathcal{A}}(H^*(ku;\mathbb{F}_2),H^*(ku;\mathbb{F}_2)) \implies \pi_{t-s}\Hom(ku^{\wedge}_2,ku^{\wedge}_2)$$
converges and has $E_2$-term isomorphic to
$$Ext^{s,t}_{E[Q_0,Q_1]}(\mathbb{F}_2,H^*(ku;\mathbb{F}_2)).$$
By either the final paragraph of \cite{AngLind} or \cite[Lemma 3.8 \& Table 3.9]{AdamsPriddy}, it is seen that there are no classes present on this $E_2$-page with $t-s=-1,-3,$ \text{ or} $-5$.
\end{proof}

\begin{lem} \label{Lemma:AdamsPriddyII}
Suppose that $E$ is a form of $K$-theory.  Then there are equivalences of spectra
$$ku^{\wedge}_{2} \simeq \tau_{\ge 0} E, \text{ and}$$
$$\Sigma^4 ku^{\wedge}_{2} \simeq \tau_{\ge 3} gl_1(E).$$
\end{lem}

\begin{proof}
Our arguments closely follow those in \cite{AdamsPriddy}.

The spectrum $\tau_{\ge 0} E$ has a Postnikov tower with primary $k$-invariants
$$\Sigma^{2k} H\mathbb{Z}_2 \longrightarrow \Sigma^{2k+3} H\mathbb{Z}_2,$$
for $k \ge 0$.  The $2$-periodicity of $E$ implies that all of these primary $k$-invariants are the same up to suspension.  Since $E$ is a Morava $E$-theory associated to a height $1$ formal group law, there is a homotopy ring map $BP \longrightarrow \tau_{\ge 0} E$ that is non-trivial on $\pi_2$.  This is enough to show (by comparison with the $k$-invariants in $BP$) that the primary $k$-invariants in $\tau_{\ge 0} E$ are all suspensions of the Milnor operation $Q_1:H\mathbb{Z}_2 \longrightarrow \Sigma^3 H\mathbb{Z}_2$.

It follows as in \cite[Proposition 2.1]{AdamsPriddy} that the Postnikov tower spectral sequence
$$E_1^{s,t} = \bigoplus_{s} H^t(\Sigma^{2s} H\mathbb{Z}_2;\mathbb{F}_2) \implies H^{*}(\tau_{\ge 0} E;\mathbb{F}_2)$$
degenerates on the $E_2$-page onto the $0$-line, proving that 
$$H^*(\tau_{\ge 0} E;\mathbb{F}_2) \cong H^*(ku;\mathbb{F}_2) \cong \mathcal{A}//E[Q_0,Q_1]$$
as modules over the Steenrod algebra $\mathcal{A}$.  Applying \cite[Theorem 1.1]{AdamsPriddy}, we conclude that $\tau_{\ge 0} E \simeq ku^{\wedge}_2.$

To prove that $$\Sigma^4 ku^{\wedge}_{2} \simeq \tau_{\ge 3} gl_1(E),$$
note that $\Omega^{\infty} \tau_{\ge 3} gl_1(E) \simeq \Omega^{\infty} \tau_{\ge 3} E \simeq BSU_{2}^{\wedge}$ and apply \cite[Corollary 1.3]{AdamsPriddy}.
\end{proof}

\begin{proof}[Proof of Theorem \ref{Thm:FormOrientation}]
We are now ready to produce the claimed $\mathbb{E}_\infty$-ring map
$$\Sigma^{\infty}_+ BU [\beta^{-1}] \longrightarrow E.$$
The universal property of $gl_1$ \cite[Theorem 5.2]{ABGHR} states that the set of $\mathbb{E}_\infty$-ring maps
$$\Sigma^{\infty}_+ BU \to E,$$
up to equivalence, is in bijection with homotopy classes of spectrum maps $\Sigma^2 ku \to gl_1(E)$.  To obtain an $\mathbb{E}_\infty$-ring map from a spectrum map, one takes $\Omega^{\infty}$, obtaining a map $$BU \to \Omega^{\infty}gl_1(E) \subset \Omega^{\infty} E,$$ and then applies the $(\Sigma^{\infty}_+, \Omega^{\infty})$ adjunction.
We would like to produce an $\mathbb{E}_\infty$-ring map
$\Sigma^{\infty}_+ BU \to E$ that sends the class $\beta \in \pi_2(\Sigma^{\infty}_+ BU)$ to an invertible class in $\pi_2(E)$. 
Equivalently, it suffices to make a spectrum map
$$\Sigma^2 ku \longrightarrow gl_1(E)$$
which sends $\beta \in \pi_2(ku)$ to a unit in $\pi_2(gl_1(E)) \cong \pi_2(E)$.  This is the same as producing a spectrum map
$$\Sigma^2 ku \longrightarrow \tau_{\ge 2} gl_1(E)$$
 solving the following lifting problem
$$
\begin{tikzcd}
&& \tau_{\ge 2} gl_1(E) \arrow{d} \\
\Sigma^2 ku \arrow[dashed]{rru} \arrow{r} & \Sigma^2 H\mathbb{Z} \arrow{r} & \Sigma^2 H\mathbb{Z}_2.
\end{tikzcd}
$$
The obstruction to making such a lift is the composite map
$$\Sigma^2 ku \longrightarrow \Sigma^{2} H\mathbb{Z} \longrightarrow \Sigma^2 H\mathbb{Z}_2 \longrightarrow \Sigma \tau_{\ge 3} gl_1(E).$$
By Lemmas \ref{Lemma:AdamsPriddy} and \ref{Lemma:AdamsPriddyII}, this obstruction is necessarily $0$.

Next, we turn to the question of producing an $\mathbb{E}_\infty$-ring homomorphism from $MUP$ to $E$.   By the theory of Thom spectra (due to \cite{LMS} and written in more modern language by \cite{ABGHR},\cite{OmarToby}), it suffices to check that the composite
$$ku \longrightarrow \text{pic}(\mathbb{S}) \longrightarrow \text{pic}(E)$$
is nullhomotopic.  In particular, that this suffices can be recovered from \cite[Lemma 3.15]{OmarToby} and the fact that, as explained in their proof of \cite[Proposition 3.16]{OmarToby}, the space they denote $B(R,A)$ is the homotopy fiber of the map $\text{Pic}(R) \to \text{Pic}(A)$.

Since $E$ is even periodic, the map $ku \longrightarrow \text{pic}(E)$ is $0$ on $\pi_0$.  This ensures the existence of a factorization
$$
\begin{tikzcd}
&& \Sigma gl_1(E) \arrow{d} \\
ku \arrow{r} \arrow[dashed]{rru}{f} & \text{pic}(\mathbb{S}) \arrow{r} & \text{pic}(E).
\end{tikzcd}
$$

Our strategy will be to prove that an arbitrary map
$$ku \stackrel{f}{\longrightarrow} \Sigma gl_1(E)$$
must be nullhomotopic.

First, note that the composite 
$$\Sigma^{4} ku \stackrel{\beta^2}{\longrightarrow} ku \stackrel{f}{\longrightarrow} \Sigma gl_1(E)$$
must factor through $\Sigma \tau_{\ge 3} gl_1(E)$.  By Lemmas \ref{Lemma:AdamsPriddy} and \ref{Lemma:AdamsPriddyII}, it is therefore nullhomotopic.  This provides a quick proof of an $\mathbb{E}_\infty$ $MSU$ orientation of $E$; we will next show an $MU$ orientation, and then finally the desired $MUP$ orientation.

Consider the composite 
$$\Sigma^2 ku \stackrel{\beta}{\longrightarrow} ku \stackrel{f}{\longrightarrow} \Sigma gl_1(E).$$
Since this map is null when pulled back to $\Sigma^4 ku$, it must in fact factor as a composite
$$\Sigma^2 ku \longrightarrow \Sigma^{2} H\mathbb{Z} \longrightarrow \Sigma gl_1(E).$$
By Lemma \ref{Lemma:AdamsPriddyII}, there is a Postnikov fiber sequence
$$\Sigma^5 ku^{\wedge}_{2} \longrightarrow \tau_{\ge 2} \Sigma gl_1(E) \longrightarrow \Sigma^3 H\mathbb{Z}_2.$$
Since there are no nontrivial maps $\Sigma^2 H\mathbb{Z} \longrightarrow \Sigma^3 H\mathbb{Z}_2$, we have a factorization 
$$
\begin{tikzcd}
&& \Sigma^5 ku_{2}^{\wedge} \arrow{d} \\
\Sigma^2 ku \arrow{r} \arrow[dashed]{urr} & ku \arrow{r}{f} & \Sigma gl_1(E),
\end{tikzcd}
$$
and we may conclude from Lemma \ref{Lemma:AdamsPriddy} that the bottom composite is null.

It remains to show that the map
$$ku \stackrel{f}{\longrightarrow} \Sigma gl_1(E)$$
is itself null.  By the above discussion, it factors as a composite
$$ku \longrightarrow H\mathbb{Z} \longrightarrow \Sigma gl_1(E).$$
There are no non-trivial maps $H\mathbb{Z} \longrightarrow \tau_{\le 1} \Sigma gl_1(E)\simeq \Sigma H\Z_2^{\times}$, and so there is a lift
$$ku \longrightarrow H\mathbb{Z} \longrightarrow \tau_{\ge 2} \Sigma gl_1(E).$$ 
Now consider the fiber sequence
$$\Sigma^5 ku^{\wedge}_{2} \longrightarrow \tau_{\ge 2} \Sigma gl_1(E) \longrightarrow \Sigma^3 H\mathbb{Z}_2.$$
While there is one possible non-trivial map $H\mathbb{Z} \longrightarrow \Sigma^3 H\mathbb{Z}_2$, that map becomes trivial when precomposed along the projection $ku \longrightarrow H\mathbb{Z}$.  It follows that the map of interest 
$$ku \stackrel{f}{\longrightarrow} \Sigma gl_1(E)$$
factors through $\tau_{\ge 4} \Sigma gl_1(E) \simeq \Sigma^5 ku_{2}^{\wedge}$, and we finish by Lemma \ref{Lemma:AdamsPriddy}.

\end{proof}

\bibliographystyle{alpha}
\bibliography{Bibliography}

\end{document}